\documentclass[11pt,twoside]{amsart}

\usepackage{amssymb,latexsym,amscd}

\newtheorem{thm}{Theorem}[section]
\newtheorem{prop}[thm]{Proposition}
\newtheorem{lem}[thm]{Lemma}
\newtheorem{cor}[thm]{Corollary}
\newtheorem{conj}[thm]{Conjecture}

\theoremstyle{definition}
\newtheorem{defn}[thm]{Definition}

\newtheorem{example}[thm]{Example}
\newtheorem{rem}[thm]{Remark}

\newtheorem{ques}[thm]{Question}

\def\cocoa{{\hbox{\rm C\kern-.13em o\kern-.07em C\kern-.13em o\kern-.15em A}}}

\theoremstyle{remark}

\newcommand{\pr}[1]{{\mathbb P}^{#1}}
\newcommand{\skipit}[1]{{}}
\newcommand{\prfend}{\hbox to7pt{\hfil}
\par\vskip-\baselineskip\hbox to\hsize
{\hfil\vbox {\hrule width6pt height6pt}}\vskip\baselineskip}

\newcommand {\HH}{{\mathcal H}}

\def\eatit#1{}

\newcommand{\myarrow}[2]{\hbox to #1pt{\hfil$\to$\hfil}{\hskip-#1pt{\raise
10pt\hbox to#1pt{\hfil$\scriptscriptstyle #2$\hfil}}}}

\begin{document}

\title{Star configurations in $\pr n$}

\author{A.\ V.\ Geramita, B.\ Harbourne \& J.\ Migliore}

\address{A.\ V.\ Geramita\\
Department of Mathematics\\
Queen's University\\
Kingston, Ontario, and\\
Dipartimento di Matematica\\ 
Universit\`a di Genova\\ 
Genova, Italia}
\email{anthony.geramita@gmail.com}

\address{Brian Harbourne\\
Department of Mathematics\\
University of Nebraska\\
Lincoln, NE}
\email{bharbour@math.unl.edu}

\address{Juan Migliore\\
Department of Mathematics\\
University of Notre Dame\\
South Bend, IN}
\email{migliore.1@nd.edu}

\date{March 25, 2012}

\markboth{A.\ V.\ Geramita, B.\ Harbourne, J.\ Migliore}{Symbolic powers of star configurations}

\thanks{Acknowledgments:  Geramita thanks 
the NSERC for research support, while Harbourne and Migliore thank the National Security Agency 
which sponsored this research
under Grant/Cooperative
agreement H98230-11-1-0139 for Harbourne and 
H98230-12-1-0204 for Migliore, and thus
the United States Government is authorized to reproduce and distribute reprints
notwithstanding any copyright notice.
Migliore  was also partially supported by a grant from the
Simons Foundation (\#208579).}

\begin{abstract}
Star configurations are certain unions of linear subspaces of projective space.  They 
have appeared in several different contexts: the study of extremal Hilbert functions for 
fat point schemes in the plane; the study of secant varieties of some classical 
algebraic varieties; the study of the resurgence of projective schemes. In this paper 
we study some algebraic properties of the ideals defining star configurations,
including getting partial results about Hilbert functions, generators and minimal free resolutions 
of the ideals and their symbolic powers.
We also show that their symbolic powers define arithmetically Cohen-Macaulay
subschemes and we obtain results about the primary decompositions of the powers
of the ideals. As an application, we compute the resurgence for the ideal of the 
codimension $n-1$ star configuration  
in $\pr{n}$ in the monomial case (i.e., when the number of hyperplanes is $n+1$).
\end{abstract}



\maketitle

\section{Introduction}

A \emph{star configuration} of codimension $c$ in $\pr n$ is a 
certain union of linear subspaces $V_1,\ldots,V_i$ each of codimension $c$.
These have arisen as objects of study in numerous research projects lately, including
\cite{PSC, BH, CChG, CHT, GMS, MN7, TVV}, but these references make use of only 
a partial understanding of the properties of star configurations.
Thus it is of interest to understand them better.

Here we study powers and symbolic powers of ideals of star configurations in $\pr n$ 
(over an algebraically closed field of arbitrary characteristic). Since the subspaces $V_i$ are distinct
with none containing any of the others, and each is a complete intersection, the $m$th symbolic power 
of the ideal $I$ of the star configuration is $I^{(m)}=I(V_1)^m\cap\cdots\cap I(V_i)^m$. 

Combinatorially equivalent collections of linear spaces can have very different algebraic properties,
as Example \ref{ACMvNotACM} shows by exhibiting two collections of lines in $\pr 3$
with the same intersection posets but where one gives an arithmetically Cohen-Macaulay
(ACM) subscheme 
and the other not. The situation with star configurations (defined below 
in terms of their intersection posets)
is very different. We will show in Proposition \ref{basic facts}
that every star configuration is ACM, with so-called generic Hilbert
function (meaning that the $h$-vector coincides with the dimension of the appropriate coordinate ring until a prescribed degree and is then zero).  We also  note that this property does not characterize 
star configurations, since (at least in codimension two) there exist unmixed configurations of linear varieties with the same Hilbert functions as star configurations, which are themselves not even ACM
(see Remark \ref{sameHF}).

We then show that  every symbolic power of the ideal of a star configuration of any codimension 
defines an ACM subscheme (see Theorem \ref{sym pwrs acm}). 
This contrasts with what we will see in Example \ref{ACMvNotACM}.
We also pose a conjecture for the primary decompositions of powers of ideals of star configurations
and in some cases verify the conjecture (see Conjecture \ref{hopefulconj} 
and Theorem \ref{monomialprdecomp}). As an application
we use Theorem \ref{monomialprdecomp} to determine  
the resurgence $\rho(I)$ of the ideal $I$ of a positive dimensional star configuration
(see Theorem \ref{resurgenceThm}). The only other exact determination of
the resurgence of a positive dimensional subscheme which is not a cone over a 
0-dimensional subscheme and for which the resurgence is bigger than 1
is that of \cite{GHVT}, using a different method.
\section{Preliminaries}

We let $R = k[x_0,\dots,x_n]$ where $k$ is an arbitrary infinite field, and where we regard
$R$ as a graded ring with the usual grading (where each variable has degree 1
and nonzero elements of $k$ have degree 0).

\begin{defn}\label{Stardef}
Let $\HH= \{H_1,\dots,H_s\}$ be a collection of $s\geq 1$ 
distinct hyperplanes in $\pr n$ corresponding to linear forms 
$L_1,\dots,L_s$. We assume that the hyperplanes \emph{meet properly}, by which we mean
that the intersection of any $j$ of these hyperplanes is either empty or has codimension $j$.  
For any $1 \leq c \leq \min(s,n)$, let $V_c(\HH,\pr n)$ be the union of the codimension $c$ linear varieties defined 
by all the intersections of these hyperplanes, taken $c$ at a time:
\[
V_c(\HH,\pr n) = \bigcup_{1 \leq i_1 < \dots < i_c \leq s} H_{i_1} \cap \dots \cap H_{i_c}.
\]
(When $\pr n$ or $\HH$ is clear from the context, we may write 
$V_c$ or $V_c(\pr n)$ or $V_c(\HH)$ in place of $V_c(\HH,\pr n)$.) 
We call $V_c$ {\em the codimension $c$ skeleton} associated to $\HH$ 
or sometimes simply a {\em codimension $c$ star configuration}.
We denote by $V_c^{(\ell)}$ the subscheme of $\pr n$ defined by the ideal 
\[
I_{V_c}^{(\ell)} = \bigcap_{1 \leq i_1 < \cdots < i_c \leq s} (L_{i_1},\ldots,L_{i_c})^\ell.
\]
Note that $I_{V_c}^{(\ell)}$ is the {\em $\ell$-th symbolic power} of $I_{V_c}$.  
\end{defn}

\begin{rem}
We are most interested in the case of star configurations in $\pr n$ for which $s\geq n+1$.
When $1\leq s\leq n$, the star is a either a linear subvariety of projective space or
a projective cone over a star in $\pr {s-1}$. But for some proofs it is convenient to allow
$s < n+1$.
\end{rem}

We now recall the following definition.  

\begin{defn}\label{ACMdef}
Let $Z\subseteq \pr n$ be a closed subscheme whose defining sheaf of ideals is ${\mathcal I}_Z$. If 
$H^i(\pr n, {\mathcal I}_Z(d))=0$ for all $d \in \mathbb Z$ and all $0<i\le \dim Z$, we say the scheme
$Z$ is {\em arithmetically Cohen-Macaulay} or {\em ACM}.  Note that 
this is equivalent to saying that the graded ring $R/I_Z$ is a Cohen-Macaulay ring \cite[Lemma 1.2.3]{migliore}.
\end{defn}

\begin{defn} For a scheme $V$ of codimension $c$ in $\mathbb P^n$ (not necessarily ACM), the {\em $h$-vector} of $V$ (or, more precisely, of $R/I_V$) is the $(n-c+1)$-st difference of the Hilbert function of $R/I_V$.
\end{defn}

\begin{example}\label{ACMvNotACM} 
Here we exhibit two subschemes of $\pr3$ consisting of linear subspaces with the same 
intersection poset, the same degree and arithmetic genus (and hence the same Hilbert polynomial), one being arithmetically Cohen-Macaulay (ACM) and the other not.  In both of these cases, \cocoa \cite{cocoa} shows that the symbolic squares of the ideals define subschemes which fail to be ACM. This shows that the ACM property for the reduced curve does not imply it for symbolic powers.

Let $Q$ be the nonsingular quadric surface in $\pr 3$. Recall that $Q$ is 
isomorphic to $\pr1\times \pr1$ and hence
has two rulings. Choose any four distinct lines $V_1,V_2,V_3, V_4$ from one of
the rulings and any four distinct lines $H_1,H_2, H_3, H_4$ from the other. Let  $p_i$ be the point where 
$H_i$ and $V_1$ meet, $i=1,2$, and let $q$ be a point on $H_2$ not on any of the other lines.

Consider the line $L$ in $\pr3$ through $p_1$ and $q$. Note that $Q$ does not contain $L$. 
We now have three subschemes: $C_1$, consisting of the reduced union of $V_1,\ldots,V_4, H_1,H_2$;
$C_2$, consisting of the reduced union of $L,V_2,V_3,V_4, H_1,H_2$; and
$C_3$, consisting of the reduced union of $V_1,V_2,V_3,V_4, H_1,H_2, H_3,H_4$.
Note for any $i$ and $j$ that $V_i \cup H_j$ is a hyperplane section of $Q$.
Thus $C_3$ is the complete intersection of $Q$ with four planes, these four being the planes determined by
the pairs of intersecting lines $(V_1,H_1)$, $(V_2,H_2)$, $(V_3,H_3)$, $(V_4,H_4)$.

Note that $C_3$ is the union of $C_1$ with the disjoint union $H_3\cup H_4$.
Since, as it is easy to see, ACM subschemes are connected, we see that $H_3\cup H_4$
is not ACM. Moreover, linked schemes are either both ACM or neither ACM \cite{migliore}.
Since $C_1$ is linked with $H_3\cup H_4$, we see that
$C_1$ is not ACM.
Let $X$ be the 
union of $V_2,V_3,V_4, H_1,H_2$, so that $X \cup L = C_2$.  We see 
that $X$ is directly linked, by the complete intersection of $Q$ and three 
planes, to $H_3$, and thus $X$ is ACM.  Also, since $L$ meets $Q$ in the two 
points $p_1,q \in X \subset Q$, the ideal $I_X + I_L$ defines the 
reduced scheme $\{p_1\} \cup \{q\}$. The latter is a complete intersection 
of type $(1,1,2)$, and $Q \notin I_L$ (where by abuse of notation, 
$Q$ also represents the quadratic form defining the quadric surface), 
so in fact $I_X + I_L$ is saturated. Then from the exact sequence
\[
0 \rightarrow I_{C_2} \rightarrow I_X \oplus I_L \rightarrow I_X + I_L \rightarrow 0,
\]
sheafifying and taking cohomology it follows immediately that the first cohomology of 
$\mathcal I_{C_2}$ is zero in all twists, so $C_2$ is ACM.
Notice that both $C_1$ and $C_2$ consist of 6 lines and thus have the same degree,
and since the intersection poset of both curves is the same, then both have the same 
arithmetic genus and hence the same Hilbert polynomial.
Checking computationally using \cocoa, we verified that the symbolic square of 
the ideal of neither curve is ACM.  
\end{example}

Recall from \cite{KMMNP} the following result.

\begin{prop}\label{bdgl}
Let $I_C$ be a saturated  ideal defining a  codimension $c$ subscheme $C\subseteq \pr n$.  
Let $I_S \subset I_C$ be an ideal which defines an ACM subscheme $S$ of codimension $c-1$.  
Let $F$ be a form of degree $d$ which is not a zerodivisor on $R/I_S$.
Consider the ideal $I' = F \cdot I_C + I_S$ and let $C'$ be the subscheme it defines.
Then $I'$ is saturated, hence equal to $I_{C'}$, and there is 
an exact sequence
\[
0 \rightarrow I_S(-d) \rightarrow I_C(-d) \oplus I_S \rightarrow I_{C'} \rightarrow 0.
\]
In particular, since $S$ is an ACM subscheme of codimension one 
less than $C$, we see that $C'$ is an ACM subscheme if and only if $C$ is.  Also, 
\[
\deg C' = \deg C + (\deg F)\cdot (\deg S).
\]
Furthermore, as sets 
on $S$, we have $C' = C \cup H_F$, where $H_F$ is the 
hypersurface section cut out on $S$ by $F$.  The Hilbert function 
$h_{C'}$ of $R/I_{C'}$ is $h_{C'}(t)=h_S(t)-h_S(t-d)+h_C(t-d)$.
\end{prop}

\begin{rem}
Under rather mild assumptions, the subscheme $C'$ obtained in Proposition \ref{bdgl} can be linked in two steps to $C$ via Gorenstein ideals, and it was in this context that it was introduced in \cite{KMMNP}.  We will not use this fact below, but it is worth noting that in the literature this construction is often referred to as {\em Basic Double G-Linkage}.
\end{rem}

As an application we use Proposition \ref{bdgl} to obtain the following result. 
For this we make a definition.

\begin{defn} Let $I$ be a nonzero homogeneous ideal in the ring $R$. We define 
$\alpha(I)$ to be the
least degree among degrees of nonzero elements of $I$.
\end{defn}

We recover the fact of \cite[Lemma 2.4.2]{BH} that the initial degree $\alpha(I_{V_c})$ 
of $I_{V_c}$ (i.e., the degree of a non-zero homogeneous element of least degree) is $s-c+1$.
We also note that in the case where the hyperplanes are defined by the $s=n+1$
coordinate variables, $V_c$ was known to be ACM (see \cite[Example 2.2(b)]{HeHi}).

\begin{prop} \label{basic facts}
Let $\HH= \{H_1,\dots,H_s\}$ be a collection of distinct hyperplanes in $\pr n$
meeting properly, and let $V_c=V_c(\HH)$. Then we have the following facts.
\begin{enumerate}
\item $V_c$ is ACM.

\item The $h$-vector of $V_c$, which has $s-c+1$ entries, is
\[
\left (1,c,\binom{c+1}{2},\dots, \binom{s-1}{c-1} \right ) = 
\left ( 1,\binom{c}{c-1}, \binom{c+1}{c-1},\dots,\binom{s-1}{c-1} \right ).
\]
Note that the last binomial coefficient occurs in degree $s-c$, and can also be written $\binom{(c-1)+(s-c)}{s-c}$. 

\item $\displaystyle \deg V_c = \binom{s}{c}$.

\item The initial degree of $I_{V_c}$ is $\alpha(I_{V_c})=s-c+1$, and all of its minimal generators occur in this degree
and are monomials in the linear forms $L_i$ defining the hyperplanes $H_i$.

\end{enumerate}
\end{prop}

\begin{proof}
Notice that (3) is trivial, and we include it only for completeness.  We proceed by induction on $c$ and 
on $s  \geq c$. For any $c$, note that if $s=c$ then $V_c$ is a complete intersection of linear forms, and 
parts (1) to (4) are trivial.  
If $c=1$ and $s$ is arbitrary, $V_1$ is the union of $s$ hyperplanes, and all four assertions are 
immediate.
Now assume that the assertion is true for codimension $c-1$ and for up to $s-1$ hyperplanes.  
Let $\mathcal H' = \{ H_1,\dots,H_{s-1} \}$ and let $\mathcal H = \mathcal H' \cup \{ H_s \}$.
By induction, $V_{c-1} (\mathcal H')$ and $V_c(\mathcal H')$ are both ACM.  
We now apply Proposition \ref{bdgl} to $S = V_{c-1}(\mathcal H')$,  $C = V_c(\mathcal H')$, 
and $F = L_s$, the defining 
polynomial of $H_s$.  Since $V_c (\mathcal H) = V_c (\mathcal H') \cup H_{L_s}$, where $H_{L_s}$ is 
the hyperplane section of $V_{c-1}$ cut out by $H_s$, we immediately have (1).  Since we have 
$I_{V_c(\mathcal H)} = L_s \cdot I_{V_c(\mathcal H')} + I_{V_{c-1}(\mathcal H')}$, and by
induction minimal sets of generators of $I_{V_c(\mathcal H')}$ and $I_{V_{c-1}(\mathcal H')}$
are monomials in the $L_i$ of degree $s-c$ and $s-c+1$, respectively, we 
see $I_{V_c(\mathcal H)}$ is also generated by monomials in the $L_i$, and that the generators all
have degree $s-c+1$, which proves (4).  

It remains to prove (2).  We use the Hilbert function part of Proposition \ref{bdgl}, still with $S = V_{c-1}
(\mathcal H')$,  $C = V_c(\mathcal H')$, and $F = L_s$.  The 
$h$-vector of $V_c (\mathcal H')$ is the $(n-c+1)$-th difference of 
$h_{V_c (\mathcal H')}$, while the $h$-vector of $V_{c-1}(\mathcal H')$ 
is the $(n-c+2)$-th difference of $h_{V_{c-1}(\mathcal H')}$. Notice that 
$d=1$ in this case (in the statement of Proposition \ref{bdgl}), so the 
portion of the formula coming from $h_s(t) - h_s(t-d)$ amounts to a 
first difference. The $h$-vector of $V_c(\mathcal H')$ is
\[
\left ( 1, c, \binom{c+1}{c-1}, \dots, \binom{s-3}{c-1}, \binom{s-2}{c-1} \right ),
\]
where the last entry is in degree $s-c-1$, and the $h$-vector of $V_{c-1}(\mathcal H')$ is
\[
\left ( 1,c-1, \binom{c}{c-2}, \dots, \binom{s-3}{c-2}, \binom{s-2}{c-2} \right ),
\]
where the last entry is in degree $s-c$. Thus the $h$-vector of $V_c$ is computed by
\[
\begin{array}{cccccccccccccccccccc}
& ( & 1, & c-1, & \binom{c}{c-2}, & \binom{c+1}{c-2}, & \dots &, \binom{s-3}{c-2}, & \binom{s-2}{c-2} & ) \\
+ & ( & & 1, & c, & \binom{c+1}{c-1}, & \dots & , \binom{s-3}{c-1} ,&  \binom{s-2}{c-1} & )
\end{array}
\]
from which the desired $h$-vector of $V_c(\mathcal H)$ follows.
\end{proof}

\begin{rem}\label{sameHF}
\begin{itemize}
\item[(a)]
The $h$-vector given in Proposition \ref{basic facts} (2) is sometimes called a {\em generic} $h$-vector,
on account of its being the $h$-vector of a generic finite set of points.  Note that the ACM property automatically implies that $V_c$ is a so-called scheme of {\em maximal rank}, i.e. that the natural restriction map $H^0(\mathcal O_{\mathbb P^n}(d)) \rightarrow H^0(\mathcal O_{V_c}(d))$ has maximal rank for all $d$.  However, even for a scheme of the right degree, having maximal rank does not imply that $R/I_{V_c}$ has generic $h$-vector.  For example, when $s=4$ and $c=2$, $V_2$ has $h$-vector $(1,2,3)$, hence degree 6.  However, a general set of six skew lines in $\mathbb P^3$ has maximal rank \cite{HH} but has $h$-vector $(1,2,3,4,0,-4)$.

\item[(b)] Notwithstanding the comment in (a), there do exist linear configurations that are not ACM but nevertheless have generic $h$-vectors.  This is an easy consequence of the construction given in \cite{MN11}, starting with a minimal curve consisting of two skew lines.  Indeed, here we sketch the argument that {\em for every codimension two generic $h$-vector $(1,2,3,\dots)$ of degree at least 6, there is a non-ACM configuration of codimension two linear varieties with the given generic $h$-vector.}  We begin with curves in $\mathbb P^3$ and proceed inductively, repeatedly applying Proposition \ref{bdgl}. Start with a curve $C_0$ consisting of two skew lines in $\mathbb P^3$. Its $h$-vector is $(1,2,-1$).  Let $S_1$ be a union of four planes, such that $S_1$ contains $C_0$.  Note that $I_{S_1}$ is generated by a form of degree 4 that is the product of four linear forms.  Let $F_1$ be a general linear form.  Then $F_1 \cdot I_{C_0} + I_{S_1}$ is the saturated ideal of  a union of six lines, $C_1$, with $h$-vector $(1,2,3)$.  For $i \geq 2$ (but not $i=1$) we obtain $C_i$ inductively from $C_{i-1}$ by taking $S_i$ to be a union of $i+2$ planes containing $C_{i-1}$, and $F_i$ in each case to be a general linear form (choosing a new $F_i$ each time), and setting $I_{C_i} = F_i \cdot  I_{C_{i-1}} + I_{S_i}$.  Then $C_i$ has $h$-vector $(1,2,3,\dots,i+2$).  We then pass to the codimension two case by taking cones.
\end{itemize}
\end{rem}

\begin{rem}\label{GenericPerfectionRem}
By Proposition \ref{basic facts} (1) and (2), the artinian reduction of 
the homogeneous coordinate ring of $V_c$ is 
$k[y_1,\ldots,y_c]/{\mathfrak m}^{s-c+1}$, where ${\mathfrak m}=(y_1,\ldots,y_c)$.
Since ${\mathfrak m}^{s-c+1}$ is generated by the maximal (i.e, $r\times r$ for $r=s-c+1$) minors of the
$r\times s$ matrix
$$
\begin{pmatrix}
y_1 & y_2 & \cdots & y_c & 0 & \cdots & 0 & 0\\
0     & y_1 & y_2 & \cdots & y_c & 0 & \cdots & 0\\
& & & \cdots & & & & \\
0     & \cdots & 0 & y_1 & y_2 & y_3 & \cdots & y_c
\end{pmatrix}
$$
and has codimension $(s-c+1-r+1)(s-r+1)=c$,
the graded Betti numbers of  the homogeneous coordinate ring of $V_c$ are those given by 
the Eagon-Northcott resolution of the maximal minors of a generic matrix
of size $r\times s$ \cite{refEaHo}.   Note however, that it is well-known that
powers of ${\mathfrak m}$ all have linear resolution, consequently the calculation
of the graded Betti numbers of the powers is straightforward.  In particular, denoting by $\mathbb E^{s,c}_\bullet$ the minimal free resolution of $I_{V_c(\mathcal H)}$, we have
\begin{equation}\label{rk}
\hbox{rk } \mathbb E^{s,c}_i = \binom{s}{s-c+i} \cdot \binom{s-c+i-1}{i-1}.
\end{equation}

\end{rem}

We will need the next result for the proof of Theorem \ref{ss is acm}.

\begin{lem} \label{contain}
For each $c$, we have $I_{V_{c-1}(\mathcal H)} \subset I_{V_c(\mathcal H)}^{(2)}$.
\end{lem}

\begin{proof}
We have to show the inclusion $V_c(\mathcal H)^{(2)} \subset V_{c-1}(\mathcal H)$ of schemes.  
Since both sides are unmixed, it is enough to do this locally.  That is, we show that every 
component of $V_c(\mathcal H)^{(2)}$ lies in $V_{c-1}(\mathcal H)$.  To do this, it is 
enough to look only at the components of $V_{c-1}$ that contain the component of 
$V_c(\mathcal H)$ in question.  Now, $V_{c-1}$ is a union of codimension $c-1$ 
linear spaces and $V_c$ is its singular locus.  In particular, each component of 
$V_c$ is the intersection of $c$ of the hyperplanes $H_i$, so there pass $c$ 
components of $V_{c-1}$ through each component of $V_c$ (take away one 
$H_i$ at a time).  It thus is enough to set $\mathcal H = \{ H_1,\dots,H_c \}$ and 
prove it in this case.  Now $I_{V_c(\mathcal H)} = \langle L_1,\dots,L_c \rangle$ 
and $I_{V_c(\mathcal H)}^{(2)} = I_{V_c(\mathcal H)}^2$.  

On the other hand, let $\mathcal H' = \{ H_1,\dots, H_{c-1} \}$ and consider the 
codimension $c-1$ complete intersection  $I_{V_{c-1}(\mathcal H')} = 
\langle L_1,\dots,L_{c-1} \rangle$.  Thanks to Proposition \ref{bdgl}, we have 
\[
I_{V_{c-1}(\mathcal H)} = L_c \cdot I_{V_{c-1}(\mathcal H')} + I_{V_{c-2}(\mathcal H')}.
\]
We can thus use induction on $c$ (the low values are easy to check), and assume 
that $I_{V_{c-2}(\mathcal H')}$ is generated by degree two products of $L_1,\dots,L_{c-1}$, 
and since $I_{V_{c-1}(\mathcal H')}$ is just the complete intersection of the linear 
forms $L_1,\dots,L_{c-1}$, we have that $I_{V_{c-1}(\mathcal H)}$ is generated 
by degree two products of $L_1,\dots,L_c$.  This implies the asserted inclusion and completes the proof.
\end{proof}


\section{Symbolic powers of ideals of star configurations}

Given the ideal $I$ of a reduced ACM subscheme consisting of a union of linear spaces of projective 
space,
it's natural to ask whether the symbolic powers of $I$ also define ACM subschemes. They clearly do if
the linear subspaces are points, but otherwise it is not always the case,  
as Example \ref{ACMvNotACM} shows.
Thus Theorem \ref{sym pwrs acm}, showing that all 
of the symbolic powers of any ideal $I_{V_c}$ of 
a star configuration $V_c$ of any codimension $c$
do define ACM subschemes, 
is all the more interesting.

\begin{thm} \label{sym pwrs acm}
Let $\mathcal H = \{ H_1,\dots, H_s \}$ be hyperplanes in $\mathbb P^n$ and let $V_c := V_c(\mathcal 
H)$ for any $c$. Every symbolic power of $I_{V_c}$ is ACM.  Furthermore, if $L_i$ is a linear form 
defining the hyperplane $H_i$, then each symbolic power of
$I_{V_c}$ is generated by monomials in the $L_i$.
\end{thm}

\begin{proof}
Say $s\leq n$. If $c=s$, then $V_c$ is a linear subvariety and a complete intersection
so the result is true.
If $c<s\leq n$, choose coordinates such that the hyperplanes $H_i$ are defined by coordinate variables
$x_1,\ldots,x_s$, and extend to a full set of coordinates $x_0,x_1,\ldots,x_n$.
Let $\Lambda$ be the linear subvariety defined by $x_0=x_{s+1}=\cdots=x_n=0$.
Then $\Lambda\cap V_c$ is a codimension $c$ star in $\Lambda\cong\pr{s-1}$, and $V_c$ is a 
projective cone over $\Lambda\cap V_c$. 
In addition to the canonical surjection $k[\pr n]\to k[\Lambda]$,
we have a non-canonical inclusion $k[\Lambda]=k[x_1,\ldots,x_s]\subseteq k[x_0,\ldots,x_n]=k[\pr n]$, 
with respect to which we have $I_{V_c}^{(m)}=I_{\Lambda\cap V_c}^{(m)}k[\pr n]$
since primary decompositions extend \cite[Exercise 4.7(iv)]{AM}.
Thus $k[\pr n]/I_{V_c}^{(m)}$
is a polynomial ring over $k[\Lambda]/I_{\Lambda\cap V_c}^{(m)}$, so the result
for $V_c\subset \pr n$ follows if and only if it follows for $\Lambda\cap V_c\subset\Lambda$.
Thus we may assume that $s\geq n+1$.

Now fix the codimension, $c$, so $V_c$ is the union of $\binom{s}{c}$ linear varieties.  First assume that $s = n+1$, so without loss of generality we may assume that $L_i = x_i$ for each $i$ (modulo $s$, so $L_s=x_0$). 

We claim that $I_{V_c}$ is the Stanley-Reisner ideal of a simplicial complex, $\Delta$, of dimension $n-c$ that is the complete simplicial complex of dimension $n-c$ on $n+1$ vertices.  To construct this simplicial complex, take for the $n+1$ vertices the $n+1$ coordinate points in $\mathbb P^n$.  For convenience of notation, we will label these points by $p_0,\dots,p_n$, and without loss of generality we will assume that the vertex labelled $p_i$ is the common intersection point of the hyperplanes defined by $x_0,\dots, \hat{x_i}, \dots,x_n$.  

The component of $V_c$ cut out by the hyperplanes  $x_{i_1} =0,\dots,x_{i_c} =0$ has dimension $n-c$.    The vertices that it does {\em not} contain are precisely $x_{i_1},\dots,x_{i_c}$; that is, this component corresponds to the face of $\Delta$ which
is the linear span of the vertices with the complementary labels.  There are $n+1-c$ such vertices, so $\Delta$ has dimension $n-c$.  By construction,  it is the complete simplicial complex of dimension $n-c$ on these vertices.  Thus by construction, the Stanley-Reisner ideal corresponding to this simplicial complex is the ideal of $V_c$.
This completes the proof of our claim.  

Recall that a simplicial complex $\Delta$ is said to be {\em pure} if all of its facets have the same dimension.  It is said to be a {\em matroid} if, for every subset $W$ of the vertices (in our case $\{p_0,\dots,p_n\}$), the restriction $\Delta_W = \{ F \in \Delta \ | \ F \subset W \}$ is a pure simplicial complex.  In our setting,  simplicial complex $\Delta$ is clearly a matroid, since the restriction is again complete.  

If $c = n$, the result clearly follows since any zero-dimensional scheme is ACM.  Thus we may assume that $c < n$, i.e. that our star configuration has dimension at least one. We now recall a key fact from \cite{MT} and \cite{V}:

\begin{quotation}
{\em Let $\Delta$ be a simplicial complex and let $I_\Delta$ be its Stanley-Reisner ideal.  Then $I_\Delta^{(\ell)}$ is Cohen-Macaulay for every $\ell \geq 1$ if and only if $\Delta$ is a matroid.}
\end{quotation}

\noindent It follows from these results that $I_{V_c}^{(\ell)}$ is Cohen-Macaulay for every $\ell$, i.e. that the corresponding schemes are ACM.

Now assume that $s > n+1$.  We still have 
$\mathcal H = \{ H_1,\dots, H_s \}$, hyperplanes in $\mathbb P^n$ where $H_i$ is 
the vanishing locus of a linear form $L_i$.  Without loss of generality we may assume that 
$L_s = x_0,L_1=x_1,\dots, L_n = x_n$.  We still denote by $V_c$ the codimension $c$ star 
configuration in $\mathbb P^n$ defined by $\mathcal H$.    Let $N = s-1$ and consider the 
star configuration $W_c \subset \mathbb P^N$ defined as in our first case above, with the 
variables $x_0,\ldots, x_N$.  

Consider the linear forms $M_{n+1} = x_{n+1} - L_{n+1}, \dots, M_{N} = x_N - L_N$.  It is clear that for an {\em ACM} subscheme $V$ of $\mathbb P^N$ meeting each of the corresponding hyperplanes, successively, in codimension 1, the {\em saturated} ideal of $I_V$ is obtained by replacing $x_i$ by $L_i$, for all $i = n+1,\dots,N$, since the ACM property and the assumption about the codimension guarantee that $M_{n+1},\dots,M_N$ are a regular sequence.  In particular, for any $i \geq n+1$, $x_i$ is replaced by $L_i$.  Thus the star configuration  $W_c$ and the schemes $W_c^{(\ell)}$ defined by its symbolic powers in $\mathbb P^N$ yield $V_c$ and the schemes $V_c^{(\ell)}$ as the result of a sequence of hyperplane sections.  Since the codimension is preserved, these hyperplane sections are all proper.  Since we have shown that $W_c^{(\ell)}$ are all ACM, the claimed result follows from the fact that the ACM property is preserved under proper hyperplane sections (see for instance \cite{migliore}).  From what we have done, the claim about the ideals is also immediate. It is also clear that $\alpha(I_{W_c^{(\ell)}})=\alpha(I_{V_c^{(\ell)}})$; we will use this in
Corollary \ref{initdegsmbpower}.
\end{proof}

Theorem \ref{sym pwrs acm} makes no assertion about the Hilbert function or the minimal free resolution (apart from its length) of the symbolic powers of the ideal of a star configuration.  In Theorem \ref{ss is acm}, only in the case of the symbolic square, we give a different proof of the fact that we obtain an ACM scheme, which allows us to describe the $h$-vector (equivalently, the Hilbert function) and the graded Betti numbers.
For the proof of the theorem, we will give an explicit construction of the symbolic square of $I_{V_c}$ for any $c$, in a way that 
makes it clear that it is ACM.  Rather than squaring $I_{V_c}$, throwing away higher codimensional 
primary components, and trying to verify that the result is ACM, we take a more direct approach. 
We construct an ideal for which it is easy to see that it is ACM, 
and then we show that this ideal is actually the symbolic square.  

We will use Proposition \ref{bdgl} with $C = {V_c}$ and $S = {V_{c-1}}$.  We will 
construct an ideal $I_{C'}$ with a special choice of $F$, so this gives right away that $C'$ is an ACM subscheme, since $C$ is.  
Furthermore, we can get the minimal free resolution of $I_{C'}$ from that of $I_C$ and $I_S$ by studying a suitable 
mapping cone.  We will then see that $C'$ is precisely the symbolic square of $C$ in this case.

\begin{thm} \label{ss is acm}
Let $\mathcal H = \{ H_1,\dots, H_s \}$ and let $V_i := V_i(\mathcal H)$ for all $i$. Then
\begin{enumerate}
\item The $h$-vector of $V_c^{(2)}$ is as follows
\[
\Delta^{n-c+1} h_{R/I_{V_c^{(2)}}}(t) = \left \{
\begin{array}{cl}
 \binom{t+c-1}{c-1} & \hbox{if $t \leq s-c$} \\ \\
 \binom{s}{c-1} & \hbox{if $s-c+1 \leq t \leq 2s-2c+1$ } \\ \\
 0 & \hbox{if $t > 2s-2c+1$}
\end{array}
\right .
\]

\item The minimal free resolution of $I_{V_c}^{(2)}$ has the form
\[
0 \rightarrow \mathbb F_c \rightarrow \cdots \rightarrow \mathbb F_1 \rightarrow \mathbb I_{V_c}^{(2)} \rightarrow 0
\]
where 
\[
\mathbb F_i = \mathbb E^{s,c}_i(-1+c-s) \oplus \mathbb E^{s,c-1}_{i-1}(-1+c-s) \oplus \mathbb E^{s,c-1}_i
\]
using the notation of Remark \ref{GenericPerfectionRem}.  In particular, 
\[
\mathbb F_i = R(-2s+2c-1-i)^{M_i} \oplus R(-s+c-1-i)^{N_i}
\]
where 
\[
M_i  =
\left \{
\begin{array}{ll}
 \binom{s}{s-c+1}   & \hbox{if } i = 1; \\ \\
 \binom{s}{s-c+i}\cdot \binom{s-c+i-1}{i-1} + \binom{s}{s-c+i} \cdot \binom{s-c+i-1}{i-2} & \hbox{if } 2 \leq i \leq c
 \end{array}
 \right.
\]
and
\[
N_i = 
\left \{
\begin{array}{ll}
\binom{s}{s-c+1+i} \cdot \binom{s-c+i}{i-1} & \hbox{if } 1 \leq i \leq c-1; \\ \\
0 & \hbox{if } i = c.
\end{array}
\right.
\]
\end{enumerate}
\end{thm}

\begin{proof}
By Lemma \ref{bdgl} (4) applied to $V_{c-1}$, $I_{V_{c-1}}$ is entirely generated in 
degree $s-c+2$,  while $I_{V_c}$ is entirely generated in degree $s-c+1$.  Let 
$F \in I_{V_c}$ be a general element of degree $s-c+1$.  Then $F$ does not 
vanish on any component of $V_{c-1}$, i.e. it is a non-zerodivisor on $R/I_{V_{c-1}}$.

As mentioned above, $V_{c-1}$ is a union of codimension $c-1$ linear spaces 
and $V_c$ is its singular locus.  In particular, each component of $V_c$ is the 
intersection of $c$ of the hyperplanes $H_i$, so there pass $c$ components of 
$V_{c-1}$ through each component of $V_c$ (take away one $H_i$ at a time).  
Since $F \in I_{V_c}$ and $F$ does not vanish on any component of $V_{c-1}$, 
the subscheme of $V_{c-1}$ cut out by $F$ thus has multiplicity at least $c$ 
locally along each component of $V_c$.  This accounts for a subscheme of 
degree at least $c \cdot \binom{s}{c}$.  On the other hand, a quick calculation shows
\[
(\deg F) \cdot (\deg V_{c-1}) = (s-c+1) \cdot \binom{s}{c-1} = c \cdot \binom{s}{c}.
\]
We conclude that $F$ cuts out a subscheme supported on $V_c \subset V_{c-1}$ with 
multiplicity exactly $c$ along each component of $V_c$.  Consequently, thanks to  
Proposition \ref{bdgl}, the subscheme defined by the ideal $F \cdot I_{V_c} + I_{V_{c-1}}$ 
is supported on $V_c$ and has degree $c+1$ along each component.  

This is the same degree and support as the scheme defined by the symbolic square of $I_{V_c}$, and 
both $I_{V_c}^{(2)}$ and $F \cdot I_{V_c} + I_{V_{c-1}}$ are unmixed (in particular, saturated).  To show 
equality, then, we just have to show one inclusion.  We will show
\addtocounter{thm}{1}
\begin{equation} \label{ideal of ss}
F \cdot I_{V_c} + I_{V_{c-1}} \subseteq I_{V_c}^{(2)} .
\end{equation}
First, any element of $F \cdot I_{V_c}$ is an element of $I_{V_c}^{(2)}$ since $F \in I_{V_c}$. 
Furthermore, by Lemma \ref{contain} we have $I_{V_{c-1}} \subset I_{V_c}^{(2)}$, so the inclusion 
follows, and the ideals are equal.
We thus have a new proof that  $V_c^{(2)}$ is ACM.  

Now we can write the Hilbert function,  a minimal generating set 
and minimal free resolution using Proposition \ref{bdgl}.  Indeed,
observe that the claimed $h$-vector is actually
\[
\Delta^{n-c+1} h_{R/I_{V_c^{(2)}}}(t) = \left \{
\begin{array}{rcll}
\Delta^{n-c+1} h_{R/I_{V_{c-1}}}(t) & =  & \binom{t+c-1}{c-1} & \hbox{if $t \leq s-c$} \\ \\
 \Delta^{n-c+1} h_{R/I_{V_{c-1}}} (s-c+1) & = & \binom{s}{c-1} & \hbox{if $s-c+1 \leq t \leq 2s-2c+1$ } \\ \\
 0 &&& \hbox{if $t > 2s-2c+1$}
\end{array}
\right .
\]
The first two lines are immediate since (\ref{ideal of ss}) shows that $I_{V_{c-1}}$ and $I_{V_c}^{(2)}$ 
agree through degree $(s-c+1) + (s-c) = 2s-2c+1$  (since Proposition \ref{basic facts} gives the initial 
degree of $I_{V_c}$ as $s-c+1$).
The third line comes from the fact that 
\[
\Delta h^{n-c+1} h_{R/I_{V_c^{(2)}}}(t) = 
[h^{n-c+1}_{R/I_{V_{c-1}}}(t) - h^{n-c+1}_{R/I_{V_{c-1}}}(t-(s-c+1))] + h^{n-c+1}_{R/I_{V_{c}}}(t-(s-c+1))  .
\]
Now, thanks to Proposition \ref{basic facts}, the third term is zero in 
degree $(s-c+1) + (s-c+1) = 2s-2c+2$.  As for the first and second 
terms, they agree in degrees $\geq (s-c+1) + (s-c+1)$, so their difference is zero in this range.

It remains to find the minimal free resolution of $I_{V_c}^{(2)}$. From Proposition \ref{bdgl} and the 
above calculations, we have the short exact sequence
\[
0 \rightarrow I_{V_{c-1}}(-1+c-s) \rightarrow I_{V_c}(-1+c-s) \oplus I_{V_{c-1}} \rightarrow I_{V_c}^{(2)}
\rightarrow 0.
\]
The minimal free resolutions of $I_{V_{c-1}}$ and of $I_{V_c}$ are described in Remark 
\ref{GenericPerfectionRem}, and in particular the equation (\ref{rk}).  A mapping cone then gives a free 
resolution of $I_{V_c}^{(2)}$, and since the resolutions of $I_{V_{c-1}}$ and of $I_{V_c}$ are linear,  it is 
immediate that there is no splitting, so this is in fact a minimal free resolution.
\end{proof}

\begin{example}
Let $n=4$, $s = 7$ and $c=3$.  The $h$-vectors of $R/I_{V_2}$ and $R/I_{V_3}$ are
\[
(1,2,3,4,5,6) \ \ \ \hbox{and} \ \ \ (1,3,6,10,15),
\]
respectively.  Let $F \in (I_{V_3})_5$.  The $h$-vector of $R/(F,I_{V_2})$ is
\[
(1,3,6,10,15,20,18,15,11,6)
\]
so using Proposition \ref{bdgl} and Proposition \ref{ss is acm},  we can compute the $h$-vector of 
$R/I_{V_3}^{(2)}$ as follows:

\begin{center}

\begin{tabular}{cccccccccccccccccccccc}
1 & 3 & 6 & 10 & 15 & 20 & 18 & 15 & 11 & 6 \\
&&&&& 1 & 3 & 6 & 10 & 15 \\ \hline
1 & 3 & 6 & 10 & 15 & 21 & 21 & 21 & 21 & 21
\end{tabular}

\end{center}

Let us now compute the minimal free resolution of $I_{V_3}^{(2)}$.  As before,  $I_{V_3}^{(2)} = 
F \cdot I_{V_3} + I_{V_2}$ and we have a short exact sequence
\[
0 \rightarrow I_{V_2}(-5) \rightarrow I_{V_3}(-5) \oplus I_{V_2} \rightarrow I_{V_3}^{(2)} \rightarrow 0.
\]
Now, because the artinian reduction of $R/I_{V_2}$ and of $R/I_{V_3}$ have generic Hilbert function, 
we know the graded Betti numbers.  Hence we have a diagram
\[
\begin{array}{ccccccccccccccccc}

&&&&& 0 \\
&& &&& \downarrow \\
&& 0 && R(-12)^{15}  & \oplus & 0  \\
&& \downarrow &&& \downarrow \\
&& R(-12)^6 && R(-11)^{35} & \oplus & R(-7)^6 \\
&& \downarrow &&& \downarrow \\
&& R(-11)^7 && R(-10)^{21} & \oplus & R(-6)^7 \\
&& \downarrow &&& \downarrow \\
0 & \rightarrow & I_{V_2}(-5) & \rightarrow & I_{V_3}(-5) & \oplus 
& I_{V_2} & \rightarrow & I_{V_3}^{(2)} & \rightarrow & 0 \\

\end{array}
\]
There is no possible splitting, so the minimal free resolution of $I_{V_3}^{(2)}$ is
\[
0 \rightarrow R(-12)^{21} \rightarrow 
\begin{array}{c}
R(-7)^6 \\
\oplus \\
R(-11)^{42}
\end{array}
\rightarrow
\begin{array}{c}
R(-6)^7 \\
\oplus \\
R(-10)^{21}
\end{array}
\rightarrow I_{V_3}^{(2)} \rightarrow 0.
\]
\end{example}

We now will consider the case of codimension 2.
In preparation for stating our results, we define some matrices.
Consider a set $\HH$ of $s>n$ hyperplanes $H_i\subset \pr n$ meeting properly,
so $V_2(\HH,\pr n)$ is the union of the $\binom{s}{2}$ codimension 2 linear spaces
of the form $H_i\cap H_j$ for $i\ne j$. Let $h_i$ be the linear
form defining $H_i$. Let $P=h_1\cdots h_s$, and let
$P_i=P/h_i$. 
Let $A_{m,n}$ be the $m\times n$ 0-matrix and
$\delta(d_1,\ldots,d_r)$ the $r\times r$ diagonal
matrix with diagonal entries $d_i$.  Furthermore, consider the $1 \times s$ matrix  $B$,
the $s\times s$ matrices $C$ and $E$ and the $s\times (s-1)$ matrix $D$, defined as follows:
\[
\begin{array}{rcl}
B & = & (-P_1\ -P_2\ -P_3\ \cdots\ -P_s), \\ \\
C & = & \delta(h_1,\ldots,h_s), \\ \\
D & = & 
\begin{pmatrix}
-h_1   & 0          &  0      &\cdots & 0 & 0\\
h_2   & -h_2          &  0      &\cdots & 0 & 0\\
0        &   h_3   &  -h_3      & \cdots & 0 & 0\\
       &     &     \cdots &  & & \\
0        &   0   &  0      & \cdots & h_{s-1} & -h_{s-1}\\
0        &   0   &  0      & \cdots & 0 & h_s
\end{pmatrix}, \text{ and}\\ \\
E & = & -\delta(P_1,\ldots,P_s).
\end{array}
\]
Finally, 
\begin{itemize}
\item when $m=2r$ is even, let $\Delta_m$ be the $(sr+1)\times sr$ matrix
\[
\Delta_m=\begin{pmatrix}
B  &           A_{1,s} & A_{1,s} & A_{1,s} & \cdots  & A_{1,s} & A_{1,s} \\
C &           E           & A_{s,s} & A_{s,s} & \cdots & A_{s,s} & A_{s,s} \\
A_{s,s} & C           & E           & A_{s,s}  & \cdots & A_{s,s}  & A_{s,s} \\
              &               &               &               & \cdots &                &               \\
A_{s,s} & A_{s,s} & A_{s,s} & A_{s,s}   & \cdots & C            & E \\
A_{s,s} & A_{s,s} & A_{s,s} & A_{s,s}   & \cdots & A_{s,s} & C 
\end{pmatrix}
\]

\item  when $m=2r+1$ is odd, let $\Delta_m$ be the $s(r+1)\times s(r+1)-1$ matrix:
\[
\begin{pmatrix}
D              & E           & A_{s,s} & A_{s,s}  & A_{s,s}  &\cdots & A_{s,s}  & A_{s,s} \\
A_{s,s-1}    & C           & E           & A_{s,s}  & A_{s,s}  &\cdots & A_{s,s} & A_{s,s}  \\
A_{s,s-1}    & A_{s,s} & C           & E           & A_{s,s}  & \cdots & A_{s,s} & A_{s,s} \\
                 &               &               &  \cdots &                &             &               & \\
A_{s,s-1}    & A_{s,s} & A_{s,s} & A_{s,s} &  A_{s,s}  & \cdots & C          & E\\
A_{s,s-1}    & A_{s,s} & A_{s,s} & A_{s,s} &  A_{s,s}  & \cdots & A_{s,s} & C
\end{pmatrix}
\]
\end{itemize}

\begin{lem}\label{maxminors}
The maximal minors of $\Delta_m$ are
$\{P^r\}$, $\{P^{r-1}P_i^2\}_{i=1}^s$, $\{P^{r-2}P_i^4\}_{i=1}^s,\ldots,\{P_i^{2r}\}_{i=1}^s$ if $m$ 
is even and $\{P^{r}P_i\}_{i=1}^s$, $\{P^{r-1}P_i^3\}_{i=1}^s,\dots, 
\{P_i^{2r+1}\}_{i=1}^s$ if $m$ is odd.
\end{lem}

\begin{proof}
The matrix $\Delta_m$ is close to being upper triangular, so the maximal minors are easy to compute with
in some cases a few row and column swaps. We leave the details to the reader.
\end{proof}

\begin{thm}\label{refBDL}
Let $\HH$ be a set of $s>n$ hyperplanes $H_i\subset \pr n$ meeting properly,
where $h_i$ is the linear form defining $H_i$. Let $I=I_{V_2(\HH,\pr n)}$, the ideal of the codimension 2 skeleton $V_2(\HH,\pr n)$.
 The Hilbert-Burch matrix for $I^{(m)}$ is
$\Delta_m$ and the generators for $I^{(m)}$ are as given in Lemma \ref{maxminors}.

\end{thm}

\begin{proof}
Let $J$ be the ideal generated by the elements listed in Lemma \ref{maxminors}.
It is easy to see that they have no common divisor and the zero-locus is
$V_2(\HH,\pr n)$. Thus by the Hilbert-Burch Theorem, $J$ defines an ACM subscheme
and the primary decomposition of $J$ consists of ideals primary for the ideals of the
components of $V_2(\HH,\pr n)$. The prime ideals corresponding to irreducible 
components of $V_2(\HH,\pr n)$ are precisely the ideals of the form\
$(h_i,h_j)$, $i\neq j$. If one localizes by inverting all $h_l$ with $l\not\in\{i,j\}$,
it is easy to check by an explicit examination of the generators given in
Lemma \ref{maxminors} that the localization $J'$ of the ideal $J$ equals the localization
of $(h_i,h_j)^m$. Thus $J$ and $I^{(m)}$ have the same primary decompositions,
so $J=I^{(m)}$, which concludes the proof.
\end{proof}

\bigskip

In the case of the codimension 2 skeleton, we now give yet another proof 
that the symbolic powers are Cohen-Macaulay, with an eye, again, to proving more than can be concluded from Theorem~\ref{sym pwrs acm}.  In fact, we will show that ideals which are ``almost" symbolic powers are also Cohen-Macaulay.

\begin{cor} \label{mixed mult codim 2}
Let $\HH$ be a set of $s>n$ hyperplanes $H_i\subset \pr n$ meeting properly,
where $h_i$ is a linear form defining $H_i$. Let $I=I_{V_2(\HH,\pr n)}$, the ideal of the codimension 2 skeleton $V_2(\HH,\pr n)$.  For $1\leq k\leq s$ and $\ell\geq 1$ arbitrary,
the schemes $W_k$ defined by the saturated ideals
\[
I_{W_k} = 
\bigcap_{1 \leq i <j \leq k} (L_i,L_j)^{\ell +2} \cap 
\bigcap_{1 \leq i \leq k  < j \leq s}  (L_i, L_j)^{\ell+1} \cap \bigcap_{k < i < j \leq s} (L_i,L_j)^\ell
\]
are all ACM.
\end{cor}

\begin{proof}
This will be a byproduct of a new proof of the Cohen-Macaulayness of the symbolic powers.  This proof is inspired by a construction used in \cite{MN7}.  That paper 
studied  {\em tetrahedral curves}, i.e.\ subschemes of $\mathbb P^3$  
defined by the intersection of powers of the ideals of the six components 
of $V_2$.  The specialization of the current theorem to $V_2 \subset \mathbb P^3$ 
was proved in that paper as a special case.

Note that we have already shown this result for $V_2^{(1)}$ and $V_2^{(2)}$.  
The idea of our proof, which worked also in \cite{MN7}, is that we can apply an 
inductive argument,  passing from $I_{V_2}^{(\ell)}$ to $I_{V_2}^{(\ell +2)}$ 
by a sequence of applications of Proposition \ref{bdgl}, thus ensuring that each resulting 
scheme along the way is ACM.  In particular, $V_2^{(\ell +2)}$ is ACM, and we have our result.

Recall that we have hyperplanes $H_1,\dots, H_s$ defined by linear forms 
$L_1,\dots, L_s$.  We begin with the ideal $I_{V_2}^{(\ell)}$.  Clearly we 
have $L_2^{\ell+1} \cdots L_s^{\ell+1} \in I_{V_2}^{(\ell)}$. 
We first claim that we have an equality of saturated ideals
\[
L_1 \cdot I_{V_2}^{(\ell)} + (L_2^{\ell +1} \cdots L_s^{\ell +1}) = 
\bigcap_{2 \leq j \leq s} (L_1,L_j)^{\ell +1} \cap \bigcap_{2 \leq i < j \leq s} (L_i,L_j)^\ell.
\]
 To see this, note first that both ideals are automatically saturated and 
 unmixed (the first comes from Proposition \ref{bdgl} and the second is an 
 intersection of saturated, unmixed ideals of the same height).  Hence 
 as before, we check that they define schemes of the same degree and 
 that there is an inclusion of one into the other.  The first ideal defines a scheme of degree 
\[
\left [ \deg (V_2^{(\ell)}) \right ] + \deg (L_1) \cdot  \left [ \deg (L_2^{\ell+1} \cdots L_s^{\ell+1}) \right ] = 
\left [ \binom{s}{2} \cdot \binom{\ell +1}{2} \right ] + (1) \cdot [(s-1)(\ell+1)]
\]
thanks to Proposition \ref{bdgl}.  The ideal on the right defines a scheme of degree 
\[
(s-1) \cdot \binom{\ell+2}{2} + \left [ \binom{s}{2} - (s-1) \right ] \cdot \binom{\ell+1}{2}  .
\]
We leave it to the reader to verify that these degrees are equal.  
Since the inclusion $\subseteq$ is clear, the claim is established.  
Note that by induction we may assume that $V_2^{(\ell)}$ is ACM, 
so the scheme defined by this new ideal is also ACM, thanks to the construction of 
Proposition \ref{bdgl}.

We now show that we can construct a sequence of ACM schemes
\[
V_2^{(\ell)} \subset W_1 \subset W_2 \subset \cdots \subset W_s = V_2^{(\ell+2)}
\]
by sequentially applying Proposition \ref{bdgl}.  What we have shown so far  
is that the scheme $W_1$ defined by the ideal
\[
L_1 \cdot I_{V_2}^{(\ell)} + (L_2^{\ell +1} \cdots L_s^{\ell +1}) = 
\bigcap_{2 \leq j \leq s} (L_1,L_j)^{\ell +1} \cap \bigcap_{2 \leq i < j \leq s} (L_i,L_j)^\ell
\]
is ACM and contains $V_2^{(\ell)}$.  Notice that $I_{W_1}$ contains the 
element $L_1^{\ell+2}L_3^{\ell+1} \cdots L_s^{\ell+1}$.

We now turn to the inductive step.  Suppose we have constructed the ACM 
scheme $W_k$ defined by the saturated ideal 
\[
I_{W_k} =
\Big(\bigcap_{1 \leq i <j \leq k} (L_i,L_j)^{\ell +2}\Big) \cap 
\Big(\bigcap_{1 \leq i \leq k  < j \leq s}  (L_i, L_j)^{\ell+1}\Big) \cap 
\Big(\bigcap_{k < i < j \leq s} (L_i,L_j)^\ell\Big)
\]
and that this ideal contains the element 
$L_1^{\ell+2}\cdots L_{k}^{\ell+2} L_{k+2}^{\ell+1} \cdots L_s^{\ell+1}$.  Notice that 
\[
\deg W_k = \binom{k}{2} \binom{\ell+3}{2} + (k)(s-k) \binom{\ell+2}{2} + \binom{s-k}{2} \binom{\ell+1}{2}.
\]
We produce the ACM scheme $W_{k+1}$ via the ideal
\[
L_{k+1} \cdot I_{W_k} + (L_1^{\ell+2}\cdots L_{k}^{\ell+2} L_{k+2}^{\ell+1} \cdots L_s^{\ell+1}).
\]
Notice that thanks to Proposition \ref{bdgl}, its degree is
\[
\deg W_k + (k)(\ell+2) + (s-k-1)(\ell+1).
\]
To prove that this  ideal is equal to
\[
I_{W_{k+1}} = 
\bigcap_{1 \leq i <j \leq k+1} (L_i,L_j)^{\ell +2} \cap 
\bigcap_{1 \leq i \leq k+1  < j \leq s}  (L_i, L_j)^{\ell+1} \cap \bigcap_{k +1< i < j \leq s} (L_i,L_j)^\ell
\]
is an elementary computation along exactly the same lines as above 
(although showing that the degrees are equal is very tedious).  
It is not hard to check that this ideal 
contains the element $L_1^{\ell+2} \cdots L_{k+1}^{\ell+2} L_{k+3}^{\ell+1} \cdots L_s^{\ell+1}$.  
Thus the inductive step works, and after $s$ steps we obtain $W_s = V_2^{(\ell+2)}$.
\end{proof}

We remark that in the case $n=3$, $s=4$ (the tetrahedral curve case), the study of when the ideals defined by 
\[
(x_1,x_2)^{\alpha_1} \cap (x_1,x_3)^{\alpha_2} \cap (x_1,x_4)^{\alpha_3} \cap 
(x_2,x_3)^{\alpha_4} \cap (x_2,x_4)^{\alpha_5} \cap (x_3,x_4)^{\alpha_6} 
\]
define ACM subschemes of $\mathbb P^3$ was begun in \cite{MN7} and 
completed in \cite{francisco}.  Corollary \ref{mixed mult codim 2} gives a 
partial extension to the codimension two case in $\pr n$.


\section{Primary decompositions of powers of ideals of star configurations and applications}

In this section we consider an important special case: star configurations 
defined by monomial ideals. Such a star configuration arises from the set of
$s=N+1$ coordinate hyperplanes in $\pr N$. As motivation, we note that given any 
codimension $c$ star configuration $V_c(\HH,\pr n)$ defined by a set $\HH= \{H_1,\dots,H_s\}$
of $s>n$ hyperplanes in $\pr n$, we have $V_c(\HH,\pr n)=V_c(\HH',\pr N)\cap L$
for an appropriate $n$-dimensional linear subspace $L\subset\pr N$, where $N+1=s$ and
$\HH'=\{H_0',\ldots,H_N'\}$ are the coordinate hyperplanes for $\pr N$. 
(In particular, define $\phi:k[\pr N]\to k[\pr n]$ by $\phi:x_i\mapsto L_{i+1}$ for $0\leq i\leq N$, where 
$x_i$ is the $i$th coordinate variable and $L_i$ is the linear form which defines $H_i$.
Then $L$ is defined by the kernel of $\phi$.)
In fact, by Theorem \ref{sym pwrs acm}, we also have 
$\phi(I_{V_c(\HH',\pr N)}^{(m)})=I_{V_c(\HH,\pr n)}^{(m)}$
for all $m\geq 1$.

We now make a conjecture on the primary decomposition of 
$I_{V_c(\HH,\pr n)}^l$, which we will verify in the monomial case
(i.e., for $I_{V_c(\HH',\pr N)}^l$; see Theorem \ref{monomialprdecomp}).

\begin{conj}\label{hopefulconj}
Let $s>n$ and let $\HH=\{H_1,\ldots,H_s\}$ be hyperplanes $H_i\subset\pr n$ meeting properly,
defined by linear forms $L_i$.
Let $M$ be the irrelevant ideal in $k[\pr n]$ and $M'$ the irrelevant ideal in $k[\pr N]$,
where $N+1=s$ with $k[\pr N]=k[x_0,\ldots,x_N]$ so $M'=(x_0,\ldots,x_N)$, and let $\HH'$ be 
the $N+1$ coordinate hyperplanes in $\pr N$.
Define $\phi:k[\pr N]\to k[\pr n]$ by $\phi:x_i\mapsto L_{i+1}$. Then 
\begin{equation*}
\begin{split}
I_{V_c(\HH,\pr n)}^l&= \phi(I_{V_c(\HH',\pr N)}^l)\\
&=\phi(I_{V_c(\HH',\pr N)}^{(l)}\cap I_{V_{c+1}(\HH',\pr N)}^{(2l)}\cap\cdots\cap I_{V_N(\HH',\pr N)}^{((N-c+1)l)}\cap (M')^{(N-c+2)l})\\
&=\phi(I_{V_c(\HH',\pr N)}^{(l)})\cap \phi(I_{V_{c+1}(\HH',\pr N)}^{(2l)})\cap\cdots\cap \phi(I_{V_N(\HH',\pr N)}^{((N-c+1)l)})\cap\phi((M')^{(N-c+2)l})\\
&= I_{V_c(\HH,\pr n)}^{(l)}\cap\cdots
\cap I_{V_n(\HH,\pr n)}^{((n-c+1)l)}\cap 
\big(\phi(I_{V_{n+1}(\HH',\pr N)}^{((n-c+2)l)})\cap
\cdots\cap \phi(I_{V_N(\HH',\pr N)}^{((N-c+1)l)})\cap \phi((M')^{(N-c+2)l})\big)\\
&= I_{V_c(\HH,\pr n)}^{(l)}\cap I_{V_{c+1}(\HH,\pr n)}^{(2l)}\cap\cdots
\cap I_{V_n(\HH,\pr n)}^{((n-c+1)l)}\cap M^{(N-c+2)l}.
\end{split}
\end{equation*}
\end{conj}

\begin{rem}\label{DiscOfConj}
Conjecture \ref{hopefulconj} is somewhat complicated so some comments may be helpful.
The point is to give primary decompositions of $I_{V_c(\HH,\pr n)}^l$
in terms of intersections of symbolic powers. Of course,
the symbolic powers are not primary, but they are by definition intersections of primary ideals;
for example, $I_{V_c(\HH,\pr n)}^{(l)}$ is the intersection of the $l$th powers of the
ideals defining the various linear components of $V_c(\HH,\pr n)$.
What is true is:
\addtocounter{thm}{1}
\begin{equation}\label{bigdisplay}
\begin{split}
I_{V_c(\HH,\pr n)}^l&= \phi(I_{V_c(\HH',\pr N)})^l=\phi(I_{V_c(\HH',\pr N)}^l)\\
&=\phi(I_{V_c(\HH',\pr N)}^{(l)}\cap I_{V_{c+1}(\HH',\pr N)}^{(2l)}\cap
\cdots\cap I_{V_N(\HH',\pr N)}^{((N-c+1)l)}\cap (M')^{(N-c+2)l})\\
&\subseteq\phi(I_{V_c(\HH',\pr N)}^{(l)})\cap \phi(I_{V_{c+1}(\HH',\pr N)}^{(2l)})\cap
\cdots\cap \phi(I_{V_N(\HH',\pr N)}^{((N-c+1)l)})\cap \phi((M')^{(N-c+2)l})\\
&= I_{V_c(\HH,\pr n)}^{(l)}\cap\cdots
\cap I_{V_n(\HH,\pr n)}^{((n-c+1)l)}\cap 
\big(\phi(I_{V_{n+1}(\HH',\pr N)}^{((n-c+2)l)})\cap
\cdots\cap \phi(I_{V_N(\HH',\pr N)}^{((N-c+1)l)})\cap \phi((M')^{(N-c+2)l})\big)\\
& \subseteq I_{V_c(\HH,\pr n)}^{(l)}\cap I_{V_{c+1}(\HH,\pr n)}^{(2l)}\cap\cdots
\cap I_{V_n(\HH,\pr n)}^{((n-c+1)l)}\cap M^{(N-c+2)l}.
\end{split}
\end{equation}
The first equality follows from Proposition \ref{basic facts}(4), the second 
since $\phi$ is a homomorphism and the third by Theorem \ref{monomialprdecomp} below.
The third line (i.e., the first inclusion) holds since the image of an intersection is always contained in
the intersection of the images (for any mapping), and the fourth line holds
since $\phi(I_{V_{c+i}(\HH',\pr N)}^{((i+1)l)})= I_{V_{c+i}(\HH,\pr n)}^{((i+1)l)}$ for each $i$
by Theorem \ref{sym pwrs acm}.
The fifth line holds since $\phi(M')=M$ and since some of the
terms in the intersection have been deleted.
Thus the conjecture is that the two inclusions are equalities. 

The conjecture that the first inclusion is an equality 
says that $\phi$ commutes with the intersections. 
Having equality would give a
primary decomposition of $I_{V_c(\HH,\pr n)}^l$. Note that the tail end of this
conjectured primary decomposition, namely
$$\phi(I_{V_{n+1}(\HH',\pr N)}^{((n-c+2)l)})\cap \phi(I_{V_{n+2}(\HH',\pr N)}^{((n-c+3)l)})
\cap\cdots\cap \phi(I_{V_N(\HH',\pr N)}^{((N-c+1)l)})\cap \phi((M')^{(N-c+2)l}),$$
is primary for the irrelevant ideal, $M$. The last line of the conjecture simply asserts that this 
irrelevant component, which is not itself in general a pure power of $M$, can 
nonetheless be replaced by a pure power of $M$. 

Finally, note that the primary decompositions proposed here need not be irredundant.
For example, when $l=1$, the last line of \eqref{bigdisplay} is contained in (hence equal to) 
$I_{V_c(\HH,\pr n)}$, hence Conjecture \ref{hopefulconj} holds for $l=1$,
but obviously the primary decomposition it gives is not minimal.
\end{rem}

\begin{rem}\label{CasesWhereConjHolds}
Here we note some cases where Conjecture \ref{hopefulconj} is known to hold. 
Conjecture \ref{hopefulconj} holds when $l=1$, as noted at the end of Remark \ref{DiscOfConj}. 
It is easy to see that Conjecture \ref{hopefulconj} holds when $c=1$, since 
$I_{V_c(\HH,\pr n)}$ is principal and $V_c(\HH,\pr n)$ is a complete intersection.
Conjecture \ref{hopefulconj} holds when $c=n$, since 
$(I_{V_c(\HH,\pr n)}^l)_t=(I_{V_c(\HH,\pr n)}^{(l)})_t$
for $t\geq \alpha(I_{V_c(\HH,\pr n)}^l)=l(s-c+1)$ by
\cite[Lemma 2.3.3(c), Lemma 2.4.2]{BH}, and hence
$I_{V_n(\HH,\pr n)}^l=I_{V_n(\HH,\pr n)}^{(l)}\cap M^{(N-c+2)l}$.
And Conjecture \ref{hopefulconj} holds when $n=N=s-1$,
by Theorem \ref{monomialprdecomp}.
\end{rem}

So now we begin a study of monomial star configurations $V_c^{(l)}(\HH',\pr N)$, where $\HH'$ 
consists of the $N+1$ coordinate hyperplanes.
Consider $k[\pr N]=k[x_0,\ldots,x_N]$. Let $p_0,\ldots,p_N$ be the coordinate vertices,
where $I_{p_i}=(\{x_j:j\neq i\})$. More generally, let $\Lambda=\langle p_{i_1},\ldots,p_{i_r}\rangle$
be the linear subspace spanned by the given points $p_{i_j}$; then 
$$I_\Lambda=(\{x_j:j\not\in \{i_1,\ldots,i_r\}\}).$$
Given any monomial $\mu=x_0^{m_0}\cdots x_N^{m_N}$, we can define its $\Lambda$-degree as
$\deg_\Lambda(\mu)=\sum_{j\not\in \{i_1,\ldots,i_r\}}m_j=\deg(\mu)-\sum_{j\in \{i_1,\ldots,i_r\}}m_j$.
Note that $\deg_\Lambda(\mu)$ is just the order of vanishing of $\mu$ on $\Lambda$
(i.e., the largest power of $I_\Lambda$ containing $\mu$). 
Let $V_c=V_c(\HH',\pr N)$ and let $I=I_{V_c}$ be its ideal. It now follows 
from the definition of symbolic power that $I^{(l)}$ is generated 
by all monomials $\mu$ such that $\deg_\Lambda(\mu)\geq l$
for all irreducible components $\Lambda$ of $V_c$. 

In the next result we determine $\alpha(I^{(l)})$. This is a special case extension of
the result $\alpha(I)=N-c+2$ given in Proposition \ref{basic facts}(4). We will use 
this extension in Corollary \ref{initdegsmbpower} to extend the determination of 
$\alpha$ given in Proposition \ref{basic facts}(4) to symbolic powers in general.

\begin{prop}\label{initdegmonomsmbpower}
Let $I$ be the ideal of $V_c=V_c(\HH',\pr N)$ where $\HH'$ consists
of the $N+1$ coordinate hyperplanes, and let $l\geq1$. Define $q$ and $r$ by writing 
$l=qc+r$ for $1\leq r\leq c$.
Then $\alpha(I^{(l)})=(q+1)(N+1)-c+r$. 
\end{prop}

\begin{proof} Let $\mu=(x_0\cdots x_N)^qx_0\cdots x_{N-c+r}$. 
Every component $\Lambda$ of $V_c$ is the span of exactly
$N-c+1$ of the coordinate vertices $p_i$. Thus 
$\deg_\Lambda(x_0\cdots x_N)=N+1-(N-c+1)=c$ and
$\deg_\Lambda(x_0\cdots x_{N-c+r})\geq (N-c+r+1)-(N-c+1)=r$, so
$\deg_\Lambda(\mu)\geq qc+r=l$. Thus $\mu\in I^{(l)}$, so
$\alpha(I^{(l)})\leq \deg(\mu)=(q+1)(N+1)-c+r$. 

To show that $\alpha(I^{(l)})\geq(q+1)(N+1)-c+r$,
it is enough to show for each monomial of degree 
$(q+1)(N+1)-c+r-1$ that there is a component $\Lambda$ of $V_c$
on which the monomial has order of vanishing less than $l$.
So let $\mu=x_0^{m_0}\cdots x_N^{m_N}$ be any monomial such that
$\deg(\mu)=(q+1)(N+1)-c+r-1=q(N+1)+(N-c+r)$. For some permutation 
$i_0,\ldots,i_N$ of the indices $0,\ldots,N$
we have $m_{i_0}\geq m_{i_1}\geq \cdots\geq m_{i_N}$. 
Let $\Lambda=\langle p_{i_0},\ldots,p_{i_{N-c}}\rangle$.
The order of vanishing of $\mu$ on $\Lambda$ is
$$\deg_\Lambda(\mu)=m_{i_{N-c+1}}+\cdots +m_{i_N}=\deg(\mu)-(m_{i_0}+\cdots+m_{i_{N-c}}).$$
This is largest when $m_{i_0}+\cdots+m_{i_{N-c}}$ is least. 
We can replace $\mu$ with $\mu'=x_0^{m_0'}\cdots x_N^{m_N'}$ of the same degree
such that still $m_{i_0}'\geq \cdots\geq m_{i_N}'$ but such that
the exponents are as close to constant as possible
(i.e., such that $m_{i_0}'-m_{i_N}'\leq 1$). Doing this increases the smaller
exponents at the expense of the larger exponents, so we have 
$\deg_\Lambda(\mu)\leq\deg_\Lambda(\mu')$. 
Since $\deg(\mu')=q(N+1)+(N-c+r)$ we see that 
$m_{i_j}=q+1$ for $j=0,\ldots, N-c+r-1$, while
$m_{i_j}=q$ for $j=N-c+r,\ldots, N$. Thus
$\deg_\Lambda(\mu)\leq \deg_\Lambda(\mu')=\deg(\mu')-(m_{i_0}+\cdots+m_{i_{N-c}})=
q(N+1)+(N-c+r)-(N-c+1)(q+1)=qc+r-1<l$.
\end{proof}

More generally we have:

\begin{cor}\label{initdegsmbpower}
Now let $V_c(\pr n)$ be the codimension $c$ skeleton for a star configuration
on $s>n$ hyperplanes in $\pr n$ and let $I$ be its ideal.
Define $q$ and $r$ by writing $l=qc+r$ for $1\leq r\leq c$.
Then $\alpha(I^{(l)})=(q+1)s-c+r$. 
\end{cor}

\begin{proof} This follows from Theorem \ref{sym pwrs acm} 
(see also the last sentence of the proof of Theorem \ref{sym pwrs acm})
and Proposition \ref{initdegmonomsmbpower}.
\end{proof}

\begin{prop}\label{finaldegmonomsmbpower}
Let $I$ be the ideal of $V_c=V_c(\HH',\pr N)$ where $\HH'$ consists
of the $N+1$ coordinate hyperplanes, and let $l\geq1$. 
Then $I^{(l)}$ is generated in degree at most $l(N-c+2)$;
more precisely, in any minimal set of homogeneous generators of $I^{(l)}$,
the degree $\omega(I^{(l)})$ of a generator of maximum degree is 
$\omega(I^{(l)})=l(N-c+2)=\alpha(I^l)$.
\end{prop}

\begin{proof} First we note that $\alpha(I^l)=l\alpha(I)=l(N-c+2)$.
The ideal $I^{(l)}$ is generated by all
monomials $\mu=x_0^{m_0}\cdots x_N^{m_N}$ such that
the $c$ smallest exponents sum to $l$. The maximum degree of such 
a monomial which is not divisible by
another such monomial is $l(N-c+2)$; take for example $\mu=(x_{c-1}\cdots x_N)^l$,
and note $\mu$ is not divisible by any other monomial in this generating set.
\end{proof}

We now prove Conjecture \ref{hopefulconj} in the monomial case.

\begin{thm}\label{monomialprdecomp}
Let $I$ be the ideal of $V_c=V_c(\HH',\pr N)$ where $\HH'$ consists
of the $N+1$ coordinate hyperplanes and $M'$ is the irrelevant ideal, and let $l\geq1$.
Then
$$I^l=I_{V_c}^{(l)}\cap I_{V_{c+1}}^{(2l)}\cap\cdots\cap I_{V_N}^{((N-c+1)l)}\cap (M')^{(N-c+2)l}.$$
\end{thm}

\begin{proof} It is enough to show both the forward containment
$I^l\subseteq I_{V_c}^{(l)}\cap I_{V_{c+1}}^{(2l)}\cap\cdots\cap I_{V_N}^{((N-c+1)l)}\cap (M')^{(N-c+2)l}$
and the reverse containment
$I^l\supseteq I_{V_c}^{(l)}\cap I_{V_{c+1}}^{(2l)}\cap\cdots\cap I_{V_N}^{((N-c+1)l)}\cap (M')^{(N-c+2)l}$.
Moreover, if we show the forward containment for $l=1$, then we clearly have equality
for $l=1$, so it follows for $l>1$ that
$$I^l=(I_{V_c}\cap\cdots\cap I_{V_N}^{(N-c+1)}\cap (M')^{N-c+2})^l
\subseteq I_{V_c}^{(l)}\cap I_{V_{c+1}}^{(2l)}\cap\cdots\cap I_{V_N}^{((N-c+1)l)}\cap (M')^{(N-c+2)l};$$
i.e., the forward containment for $l=1$ implies the reverse containment for $l=1$
and it implies the forward containments for all $l>1$.

So now we verify the forward containment for $l=1$. As noted in the proof
of Proposition \ref{finaldegmonomsmbpower}, $I$ is generated by all monomials
$\mu=x_0^{m_0}\cdots x_N^{m_N}$ such that
the $c$ smallest exponents sum to $l=1$. We also know that $I$ is generated by
monomials of degree $N-c+2$. Thus exactly $c-1$ of the exponents $m_i$ must be 0, so the other
$N-c+2$ must be equal to 1. I.e., $I$ is generated by the square-free monomials of degree $N-c+2$, so
each $\mu$ is of the form $\mu=x_{i_0}\cdots x_{i_{N-c+1}}$ for some 
indices $0\leq i_0<\cdots<i_{N-c+1}\leq N$.
Thus it is enough to show for every square-free monomial $\mu$ of degree $N-c+2$
that $\mu\in I_{V_{c+i}}^{(i+1)}$ for $i=0,\ldots, N-c+1$ and that $\mu\in (M')^{N-c+2}$.
Clearly we have $\mu\in (M')^{N-c+2}$, so consider $\mu\in I_{V_{c+i}}^{(i+1)}$.
We must check that $\deg_\Lambda(\mu)\geq i+1$ for each component 
$\Lambda$ of $V_{c+i}$. But $\Lambda$ is spanned by exactly $N-c-i+1$ coordinate
vertices, hence $\deg_\Lambda(\mu)\geq \deg(\mu)-(N-c-i+1)=i+1$, as needed.

We now verify the reverse inclusion. Let $\mu=x_0^{m_0}\cdots x_N^{m_N}$ and assume
$$\mu\in I_{V_c}^{(l)}\cap I_{V_{c+1}}^{(2l)}\cap
\cdots\cap I_{V_N}^{((N-c+1)l)}\cap (M')^{(N-c+2)l}\eqno{({}^{**})}.$$
We will show that $\mu\in I^l$. For simplicity we demonstrate the argument only in case
$m_0\geq m_1\geq\cdots\geq m_N$; up to a permutation of the indices, the general argument is the same.
Our proof will be by induction on $l$, the case $l=1$ having been established above.

If $m_{N-c+1}\geq l$, then $(x_0\cdots x_{N-c+1})^l$ divides $\mu$, but $x_0\cdots x_{N-c+1}\in I$,
so $\mu\in I^l$. Now assume $m_{N-c+1}<l$. In any case we have $m_{N-c+1}>0$, since
if $m_{N-c+1}=0$, then $\mu$ is not divisible by any square-free monomial of degree $N-c+2$ and hence
$\mu\not\in I$, but by assumption $({}^{**})$, $\mu\in I^{(l)}\subseteq I$. In particular, $\mu$ is divisible
by $x_0\cdots x_{N-c+1}$; let $\mu'=\mu/(x_0\cdots x_{N-c+1})$.
If we check that 
$$\mu'\in I_{V_c}^{(l-1)}\cap I_{V_{c+1}}^{(2(l-1))}\cap
\cdots\cap I_{V_N}^{((N-c+1)(l-1))}\cap (M')^{(N-c+2)(l-1)};$$
then $\mu'\in I^{l-1}$ by induction, so $\mu=\mu'x_0\cdots x_{N-c+1}\in I^l$,
as claimed. We will use the following function. Given distinct elements
$j_1,\ldots,j_r\in\{0,\ldots,N\}$ and $0\leq t\leq N$, 
let $\nu_{j_1,\ldots,j_r}(t)=|\{0,\ldots,t\}\cap\{j_1,\ldots,j_r\}|$.
Thus, for example, $\nu_j(t)$ is 1 if $0\leq j\leq t\leq N$ and 
$\nu_j(t)$ is 0 if $0\leq t< j\leq N$. 

We first check that $\mu'\in (M')^{(N-c+2)(l-1)}$.
Since $\mu\in (M')^{(N-c+2)l}$, we have $\deg(\mu)\geq (N-c+2)l$, so
$\deg(\mu')\geq (N-c+2)l-(N-c+2)=(N-c+2)(l-1)$, hence $\mu'\in (M')^{(N-c+2)(l-1)}$.

Now we check that $\mu'\in I_{V_N}^{((N-c+1)(l-1))}$. 
It suffices to check that $\deg_{\langle p_i\rangle}(\mu')\geq (N-c+1)(l-1)$ for each $i$,
where $p_0,\ldots,p_N$ are the coordinate vertices.
For all $i$ we have
$$\deg_{\langle p_i\rangle}(\mu')=\deg_{\langle p_i\rangle}(\mu)-(N-c+2-\nu_i(N-c+1)).$$
If $i\leq N-c+1$, then $\nu_i(N-c+1)=1$ so using 
$\deg_{\langle p_i\rangle}(\mu)\geq (N-c+1)l$ (which we have since $\mu\in I_{V_N}^{((N-c+1)l)}$) 
we obtain 
$$\deg_{\langle p_i\rangle}(\mu)-(N-c+2-\nu_i(N-c+1))\geq (N-c+1)l-(N-c+1)=(N-c+1)(l-1).$$
If $i>N-c+1$, then $\nu_i(N-c+1)=0$ and
\begin{equation*}
\begin{split}
\deg_{\langle p_i\rangle}(\mu')&=\deg_{\langle p_i\rangle}(\mu)-(N-c+2-\nu_i(N-c+1))\\
                                                       &=\deg(\mu)-m_i-(N-c+2)\\
                                                       &\geq (N-c+2)l-m_i-(N-c+2)\\
                                                       &\geq (N-c+2)l-(l-1)-(N-c+2)\\
                                                       &= (N-c+1)(l-1),
\end{split}
\end{equation*}
where the fourth line uses the assumption that $m_{N-c+1}<l$.

Now we check that $\mu'\in I_{V_{N-1}}^{((N-c)(l-1))}$. 
Let $p_{i_1}$ and $p_{i_2}$ be arbitrary distinct coordinate vertices,
and assume $i_1<i_2$.
It suffices to check that $\deg_{\langle p_{i_1},p_{i_2}\rangle}(\mu')\geq (N-c)(l-1)$.
For all $i$ we have
$$\deg_{\langle p_{i_1},p_{i_2}\rangle}(\mu')=
\deg_{\langle p_{i_1},p_{i_2}\rangle}(\mu)-(N-c+2-\nu_{i_1,i_2}(N-c+1)).$$
If $i_2\leq N-c+1$, then $\nu_{i_1,i_2}(N-c+1)=2$ so using
$\deg_{\langle p_{i_1},p_{i_2}\rangle}(\mu)\geq (N-c)l$ we have
$$\deg_{\langle p_{i_1},p_{i_2}\rangle}(\mu)-(N-c+2-\nu_{i_1,i_2}(N-c+1))
\geq (N-c)l-(N-c)=(N-c)(l-1).$$
If $i_1\leq N-c+1<i_2$, then $\nu_{i_1,i_2}(N-c+1)=1$ so using 
$\deg_{\langle p_{i_1},p_{i_2}\rangle}(\mu)=
\deg_{\langle p_{i_1}\rangle}(\mu)-m_{i_2}\geq (N-c+1)l-m_{i_2}\geq (N-c+1)l-(l-1)$ gives
$$\deg_{\langle p_{i_1},p_{i_2}\rangle}(\mu)-(N-c+2-\nu_{i_1,i_2}(N-c+1))
\geq (N-c+1)l-(l-1)-(N-c+1)=(N-c)(l-1).$$
If $N-c+1<i_1$, then $\nu_{i_1,i_2}(N-c+1)=0$ so using 
$\deg_{\langle p_{i_1},p_{i_2}\rangle}(\mu)=
\deg(\mu)-m_{i_1}-m_{i_2}\geq (N-c+2)l-2(l-1)=(N-c)l+2$ gives
$$\deg_{\langle p_{i_1},p_{i_2}\rangle}(\mu)-(N-c+2-\nu_{i_1,i_2}(N-c+1))
\geq (N-c)l+2-(N-c+2)=(N-c)(l-1).$$

Now we must check that $\mu'\in I_{V_{N-2}}^{((N-c-1)(l-1))}$, and then
$\mu'\in I_{V_{N-3}}^{((N-c-2)(l-1))}$, etc., but the argument follows the same pattern
of checking cases depending on how many of the indices of 
$\langle p_{i_1},\ldots,p_{i_r}\rangle$ are less than or equal to $N-c+1$,
and each case is verified in the same way as indicated above.
\end{proof}

We can partially extend this to the non-monomial case.
Given a homogeneous ideal $J$ in a polynomial ring,
we denote the saturation of $J$ by $\operatorname{sat}(J)$, meaning
the intersection of the primary components of $J$ excluding the component primary
to the irrelevant ideal (if there is one).

\begin{cor}\label{prdecompcor}
Let $I\subset k[\pr n]=R$ be the ideal of $V_c=V_c(\HH,\pr n)$ where $\HH$ consists
of $s>n$ hyperplanes $H_1,\ldots,H_s$ meeting properly where 
$M$ is the irrelevant ideal, and let $l\geq1$.
Then
$$\operatorname{sat}(I^l)=I_{V_c}^{(l)}\cap I_{V_{c+1}}^{(2l)}\cap\cdots\cap I_{V_n}^{((n-c+1)l)}.$$
\end{cor}

\begin{proof}
Since 
$I^l\subseteq I_{V_c}^{(l)}\cap I_{V_{c+1}}^{(2l)}\cap\cdots\cap I_{V_n}^{((n-c+1)l)}$
by Remark \ref{DiscOfConj} but the latter is saturated, we at least have
$\operatorname{sat}(I^l)\subseteq I_{V_c}^{(l)}\cap I_{V_{c+1}}^{(2l)}\cap\cdots\cap I_{V_n}^{((n-c+1)l)}$.
Since $I^l$ is homogeneous, the associated primes and their primary components
are homogeneous also \cite[Theorem 9, p. 153]{ZS}. Thus to show equality it suffices to
show equality after localizing for every prime ideal of the form $I_p$ for $p\in V_c$.
But after such a localization, every hyperplane $H_i$ not passing through $p$ becomes
a unit and hence $I^lR_{I_p}$ is generated by monomials in the linear forms $L_j$ for
all $H_j$ passing through $p$. Say that these $H_j$ are
$H_{j_0},\ldots, H_{j_r}$ and pick any other $n-r$ of the hyperplanes $H_i$ to obtain
$H_{j_0},\ldots, H_{j_n}$. After a change of coordinates we may assume $H_{j_i}=x_i$ 
for $i=0,\ldots,n$. Let $\HH'=\{x_0,\dots,x_n\}$, let $V_i'=V_i(\HH',\pr n)$ for all $i$ and
let $J=I_{V_c'}$. Clearly $p\in V_c'\subseteq V_c$ and $I^lR_{I_p}=J^lR_{I_p}$.
We know the primary decomposition of $J^l$ and hence of $J^lR_{I_p}$ from
Theorem \ref{monomialprdecomp}; i.e., we have
\begin{equation*}
\begin{split}
I^lR_{I_p}&=J^lR_{I_p}\\
                  &=(I_{V_c'}^{(l)}\cap I_{V_{c+1}'}^{(2l)}\cap\cdots\cap I_{V_n'}^{((n-c+1)l)}\cap M^{(n-c+2)l})R_{I_p}\\
                  &=I_{V_c'}^{(l)}R_{I_p}\cap I_{V_{c+1}'}^{(2l)}R_{I_p}\cap\cdots\cap I_{V_n'}^{((n-c+1)l)}R_{I_p}\\
                  &=I_{V_c}^{(l)}R_{I_p}\cap I_{V_{c+1}}^{(2l)}R_{I_p}\cap\cdots\cap I_{V_n}^{((n-c+1)l)}R_{I_p}\\
                  &=(I_{V_c}^{(l)}\cap I_{V_{c+1}}^{(2l)}\cap\cdots\cap I_{V_n}^{((n-c+1)l)})R_{I_p}.
\end{split}
\end{equation*}
Since this holds for all $p\in V_c$, we have 
$$\operatorname{sat}(I^l)=I_{V_c}^{(l)}\cap I_{V_{c+1}}^{(2l)}\cap\cdots\cap I_{V_n}^{((n-c+1)l)}.$$
as claimed.
\end{proof}

As an application we apply our results to compute the resurgence for certain subschemes.
We first recall the definition of the resurgence \cite{BH}. The point of the resurgence
is to provide an asymptotic measure of how far symbolic powers deviate from ordinary 
powers of the same ideal. This is not interesting in the case of an ideal $I$ if 
$I=(0)$ or $I=(1)$, so we do not define the resurgence in those cases.

\begin{defn} Let $(0)\neq I\subsetneq k[\pr n]$ be a homogeneous ideal.
The {\em resurgence} of $I$, denoted $\rho(I)$ is
$$\sup\Big\{\frac{m}{r}: I^{(m)}\not\subseteq I^r\Big\}.$$
\end{defn}

We always have $1\leq \rho(I)\leq n$. 
(Since $\alpha(I^{(m)})\leq\alpha(I^m)=m\alpha(I)$ and $r\alpha(I)=\alpha(I^r)$, 
we see that $m<r$ implies $\alpha(I^{(m)})<\alpha(I^r)$
and hence $I^{(m)}\not\subseteq I^r$. It follows that
$1\leq \rho(I)$, and by applying the main result of \cite{HoHu} we know that $\rho(I)\leq n$.)
However, it is in general quite difficult to compute the resurgence exactly, and only a few cases
have been done, so finding methods to provide exact determinations
in additional cases is of substantial interest. When $I$ is the ideal of a complete intersection,
$I^{(m)}=I^m$ for all $m\geq 0$ \cite[Lemma 5, Appendix 6]{ZS} so $\rho(I)=1$.
Moreover, $\rho(I)=\rho(I')$ if $I\subset k[\pr n]\subseteq k[\pr N]$ with $I'=Ik[\pr N]$ \cite[Proposition 2.5.1(a)]{BH},
so if the resurgence is known for a subscheme it is known for projective cones over the subscheme.
Some exact values of $\rho(I)$ are known when $I$ defines a 0-dimensional
subscheme \cite{BH, BH2, refDJ}. 
For example, if $\HH$ consists of $s>N$ hyperplanes in $\pr N$ meeting properly, then \cite{BH} shows that
$$\frac{c(s-c+1)}{s}\leq \rho(I_{V_c(\HH,\pr N)})$$
with equality in case $c=N$.
Thus when $s=N+1$ we have 
$\frac{c(N-c+2)}{N+1}\leq \rho(I_{V_c(\HH,\pr N)})$
with $\rho(I_{V_N(\HH,\pr N)})=2N/(N+1)$ when $c=N$.
We will show equality also holds when $c=N-1$, giving 
$\rho(I_{V_N(\HH,\pr N)})=3(N-1)/(N+1)$.
The only exact determinations up to now
for subschemes which are not complete intersections nor are 0-dimensional
nor are cones over such and for which the resurgence is bigger than 1
are for certain smooth unions of lines in projective space \cite{GHVT}.

\begin{thm}\label{resurgenceThm}
Let $N\geq 3$ and let $I\subset k[\pr N]$ be the ideal of $V_{N-1}=V_{N-1}(\HH',\pr N)$ where $\HH'$ consists
of the $N+1$ coordinate hyperplanes, which we denote as $H_1,\ldots,H_{N+1}$.
Then $$\rho(I)=\frac{3(N-1)}{N+1}.$$
Moreover, given $m,r\geq 1$, we have $I^{(m)}\not\subseteq I^r$ if and only if
$$\frac{m}{r} < \Big(3-\frac{2N-4}{(N-1)r}\Big)\frac{N-1}{N+1}.$$
\end{thm}

\begin{proof} Assume $k[\pr N]=k[x_0,\ldots,x_N]$.
Let $M$ be the irrelevant ideal and let $J=I_{V_N}$. By Theorem \ref{monomialprdecomp}
$$I^r=I_{V_{N-1}}^{(r)}\cap I_{V_N}^{(2r)}\cap M^{3r}=I^{(r)}\cap J^{(2r)}\cap M^{3r}.$$
Thus $I^{(m)}$ fails to be contained in $I^r$ if and only if either
$$I^{(m)}\not\subseteq I^{(r)},\ I^{(m)}\not\subseteq J^{(2r)}\text{\ or\ }I^{(m)}\not\subseteq M^{3r},$$
so
$$\rho(I)=\max\big(\sup\{m/r:I^{(m)}\not\subseteq I^{(r)}\},\sup\{m/r:I^{(m)}\not\subseteq 
J^{(2r)}\},\sup\{m/r:I^{(m)}\not\subseteq M^{3r}\}\big).$$
Since $I^{(m)}\not\subseteq I^{(r)}$ if and only if $m<r$, we have $\sup\{m/r:I^{(m)}\not\subseteq I^{(r)}\}\leq 1$.

Next, $I^{(m)}\not\subseteq J^{(2r)}$ if and only if there is a monomial $x_0^{m_0}\cdots x_N^{m_N}$
in $I^{(m)}$ but not in $J^{(2r)}$. After a permutation of the indices if need be, this condition is equivalent to there being
exponents $m_0\geq \cdots \geq m_N$ such that $m_2+\cdots+m_N\geq m$ but $m_1+\cdots+m_N< 2r$. 
Let $q=\lfloor m/(N-1)\rfloor$ and $r=m-(N-1)q$ so $m=(N-1)q+r$ and $0\leq r<N-1$.
Let $m_0'=m_1'=m_2'$, and if $r=0$, let $m_2'=\cdots=m_N'=q$, while if $r>0$, let
$m_2'=\cdots=m_{r+1}'=q+1$, and $m_{r+2}'=\cdots=m_N'=q$. 
Note that $m_2\geq \lceil(m_2+\cdots+m_N)/(N-1)\rceil\geq \lceil m/(N-1)\rceil$ so $m_2\geq m_2'$, hence
$m_1\geq m_2\geq m_2'=m_1'$. Then 
$m_0'\geq \cdots\geq m_N'$ with $m_2'+\cdots+m_N'=(N-1)q+r=m$ and 
$m_1'+\cdots+m_N'=m_1'+m_2'+\cdots+m_N'=m_1'+m\leq m_1+m_2+\cdots+m_N<2r$.
Thus $\mu'=x_0^{m_0'}\cdots x_N^{m_N'}\in I^{(m)}\setminus J^{(2r)}$,
and we have $m_0'-m_N'\leq 1$; in particular, 
each $m_i'$ is either $\lceil m/(N-1)\rceil$ or $\lfloor m/(N-1)\rfloor$
(and necessarily $m_2'=\lceil m/(N-1)\rceil$ and $m_N'=\lfloor m/(N-1)\rfloor$).
The condition that $m_1'+\cdots+m_N'< 2r$ can now be stated
as $m+\lceil m/(N-1)\rceil<2r$, or equivalently as $mN/(N-1)=m+m/(N-1)\leq 2r-1$; i.e., 
$I^{(m)}\not\subseteq J^{(2r)}$ if and only if 
$$\frac{m}{r}\leq (N-1)\frac{2-\frac{1}{r}}{N}.$$
Thus $\sup\{m/r:I^{(m)}\not\subseteq J^{(2r)}\}\leq 2(N-1)/N$.

Finally, $I^{(m)}\not\subseteq M^{3r}$ if and only if $\alpha(I^{(m)})<3r$. By Proposition \ref{initdegmonomsmbpower},
$\alpha(I^{(m)})=(q+1)(N+1)-(N-1)+r$, where $m=q(N-1)+r$ for $1\leq r\leq N-1$.
Note that 
$$(q+1)(N+1)-(N-1)+r=m+2q+2=m+2(m-r)/(N-1)+2=m+2+2\Big\lfloor \frac{m-1}{N-1}\Big\rfloor.$$
Thus $m+2+2\lfloor \frac{m-1}{N-1}\rfloor=\alpha(I^{(m)})<3r$ holds if and only if 
$m+2+2\frac{m-1}{N-1}<3r$, which simplifies to
$$\frac{m}{r} < \Big(3-\frac{2N-4}{(N-1)r}\Big)\frac{N-1}{N+1}.$$
The supremum of the right hand side over all values of $r\geq 1$ is $3\frac{N-1}{N+1}$.
Since $3\frac{N-1}{N+1}$ is greater than either 1 or $2(N-1)/N$, we see that
$\rho(I)\leq 3\frac{N-1}{N+1}$.
To show that we actually have equality, let $m=3(N-1)^2t$ and let $r=(N^2-1)t+N-1$.
Then $m+2+2\frac{m-1}{N-1}<3r$ holds (it simplifies to $(3N-8)N+7>0$), and
$$\frac{m}{r}=\frac{3(N-1)^2t}{(N^2-1)t+N-1}=\frac{3(N-1)}{N+1+\frac{1}{t}}$$
has supremum $3(N-1)/(N+1)$, taken over all $t\geq 1$.

We now have 
$$\rho(I)=\max\Big(\frac{3(N-1)}{N+1}, \frac{2(N-1)}{N}, 1\Big)=\frac{3(N-1)}{N+1}.$$
We close by proving that $I^{(m)}\not\subseteq I^r$ if and only if
$$\frac{m}{r} < \Big(3-\frac{2N-4}{(N-1)r}\Big)\frac{N-1}{N+1}.$$
From our work above we have
$I^{(m)}\not\subseteq I^r$ if and only if either
$$\frac{m}{r}<1\text{ or }\frac{m}{r}\leq (2-\frac{1}{r})\frac{N-1}{N}\text{ or }\frac{m}{r}<\Big(3-\frac{2N-4}{(N-1)r}\Big)\frac{N-1}{N+1}.$$
But $1\leq (2-\frac{1}{r})\frac{N-1}{N}<\big(3-\frac{2N-4}{(N-1)r}\big)\frac{N-1}{N+1}$ for $r\geq 2$,
so the three inequalities are subsumed by the last one when $r\geq 2$, while when $r=1$
it is enough to note that $\big(3-\frac{2N-4}{(N-1)}\big)\frac{N-1}{N+1}=1$.
\end{proof}

One of the things our results suggest is that the nice properties of star configurations generally 
may derive from the nice behavior coming from stars configurations whose ideals are monomial ideals.
As we have seen, a codimension $c$ star $V_c(\pr n)$ coming from $s$ hyperplanes in $\pr n$ is, as a point set,
the intersection with an appropriate linear space $L\subset \pr N$ of dimension $n$ 
of the codimension $c$ star $V_c(\pr N)$ coming from the $N+1$ coordinate hyperplanes in $\pr N$,
where $N+1=s$. Thus it is reasonable to ask the following question.
\begin{ques}\label{rhoquest}
If $\HH$ is a set of $s>n$ hyperplanes in $\pr n$ meeting properly and 
$\HH'$ is the set of coordinate hyperplanes in $\pr N$ for $N=s-1$,
is it true that $\rho(I_{V_c(\HH,\pr n)})=\rho(I_{V_c(\HH',\pr N)})$?
\end{ques}

We do not know the answer, but 
we at least have $\rho(I_{V_c(\HH,\pr n)})\leq\rho(I_{V_c(\HH',\pr N)})$.
(This is because if $(I_{V_c(\HH',\pr N)})^{(m)}\subseteq (I_{V_c(\HH',\pr N)})^r$,
then $(I_{V_c(\HH,\pr n)})^{(m)}\subseteq (I_{V_c(\HH,\pr n)})^r$, since by Theorem \ref{sym pwrs acm}
and its proof we have
$(I_{V_c(\HH,\pr n)})^{(m)}=((I_{V_c(\HH',\pr N)})^{(m)}+J)/J$ and 
$(I_{V_c(\HH',\pr N)})^r=((I_{V_c(\HH,\pr n)})^r+J)/J$, where $J$ is 
an ideal generated by linear forms, these forms being the ones defining the 
linear space whose intersection with $V_c(\HH',\pr N)$ gives $V_c(\HH,\pr n)$.)
In addition, Theorem \ref{resurgenceThm} shows that 
Question \ref{rhoquest} 
is true when $c=n=N-1$ and $s=n+2$: using 
$\rho(I_{V_n(\pr n)})=n(s-n+1)/s$ \cite{BH}, we have $\rho(I_{V_n(\pr n)})=n(s-n+1)/s=3(N-1)/(N+1)
=\rho(I_{V_{N-1}(\pr N)})$.


\end{document}